\documentclass[12pt]{amsart}

\usepackage[utf8]{inputenc}
\usepackage[T1]{fontenc}
\usepackage{graphicx}
\usepackage{longtable}
\usepackage{hyperref}
\usepackage{verbatim}
\usepackage{xy}
\usepackage{amssymb}
\usepackage{amsmath}

\title{From Parking Functions to Gelfand Pairs}
\date{\today}

\newtheorem{Theorem}[equation]{Theorem}
\newtheorem{Corollary}[equation]{Corollary}
\newtheorem{Lemma}[equation]{Lemma}
\newtheorem{Proposition}[equation]{Proposition}

\theoremstyle{definition}

\newtheorem{Example}[equation]{Example}

\newtheorem{Conjecture}[equation]{Conjecture}

\theoremstyle{remark}

\newtheorem{Remark}[equation]{Remark}
\newtheorem{Fact}[equation]{Fact}

\numberwithin{equation}{section}
\numberwithin{figure}{section}

\newcommand{\C}{{\mathbb C}}
\newcommand{\Z}{{\mathbb Z}}

\newcommand{\N}{{\mathbb N}}
\newcommand{\F}{{\mathbb F}}
\newcommand{\Ind}[3]{\text{Ind}^{#1}_{#2}(#3)}

\def\k{{\mathbf k}}
\newcommand{\Mk}{\C M (\k)}

\DeclareMathOperator{\id}{{\mathbf 1}}

\newcommand{\vG}{\varGamma}

\newcommand{\wvG}[1]{\widetilde{ \vG^{#1}}}

\newcommand{\mf}[1]{\mathbf{#1}}

\def\Diag{\Delta}
\def\L{{\mathcal L}}
\def\inv{^{-1}}
\def\GAP{{\sc GAP}}

\DeclareMathOperator{\Orb}{Orb}

\DeclareMathOperator{\pf}{PF}

\DeclareMathOperator{\homself}{End}

\begin{document}

\author{K\"ur\c{s}at Aker, \\ Mah\.ir B\.ilen Can}

 \address{K\"ur\c{s}at Aker: Feza G\"ursey Institute, \.Istanbul.}
 \email{aker@gursey.gov.tr}

 \address{Mahir Bilen Can: Tulane University, New Orleans.}
\email{mcan@tulane.edu}


\maketitle

\setcounter{tocdepth}{3}

	\begin{abstract}
	A pair $(G,K)$ of a group and its subgroup is called a Gelfand pair
	if the induced trivial representation of $K$ on $G$ is multiplicity free.
	Let $(a_j)$ be a sequence of positive integers of length $n$, and 
	let $(b_i)$ be its non-decreasing rearrangement. The sequence $(a_i)$ 
	is called a parking function of length $n$ if $b_i \leq i$ for all $i=1,\dots,n$.
	In this paper we study certain Gelfand pairs in relation with parking functions.
	In particular, we find explicit descriptions of the decomposition of the 
	associated induced trivial representations into irreducibles. 
	We obtain and study a new $q$ analogue of the Catalan numbers 
	$\frac{1}{n+1}{ 2n \choose n }$, $n\geq 1$. 
	\end{abstract}

	\section{\textbf{Introduction}}
	\label{sec-0}

	For a given natural number $n$ and 
	a sequence of positive integers $(a_j)$, let $(b_i)$ denote
	non-decreasing rearrangement of $(a_j)$. Then, the sequence $(a_i)$ 
	is called a \emph{parking function} of length $n$ if $b_i \leq i$.
	
	Parking functions, first introduced in a paper of Pyke \cite{Pyke}, Konheim and Weiss \cite{KW}
	in relation with a hashing problem is related to various topics in algebraic combinatorics; 
	labelled trees, hyperplane arrangements, noncrossing set partitions, diagonal harmonics...  
	For more, see the solution to the problem 5.49 in \cite{Stanley2} and the references therein. 
	
	A clever argument by Pollak counts the number of parking functions of length $n$ by constructing 
	a one-to-one correspondence between the set of parking functions and the group 
	$\Z^n_{n+1}/\Diag\Z_{n+1}^n$, 
	where  
	$\Z_{n+1}:= \Z / (n+1) \Z$ denotes the cyclic group of order $n+1$ and
	$\Diag\Z_{n+1}^n$ stands for the diagonal in 
	$
	\Z_{n+1}^n = \Z_{n+1} \times \cdots \times \Z_{n+1}.
	$ 
	This correspondence sends a parking function $(a_i)$ to the element $(a_i -1 \ \text{mod}\ n+1)$.
	Consequently, the number of parking functions of length $n$ is 
	$
	P_n:=(n+1)^{n-1}.
	$

	In this paper, we view the set of parking functions of length $n$ as an abelian group with the symmetric group $S_n$ action on it
	and investigate the representation theoretical and combinatorial consequences of this observation. In the representation theory side, it
	leads us to the realm of Gelfand pairs. In the combinatorics side, it motivates us to define a new $q$ analogue of the Catalan numbers.


%

	\subsection{Outline}	
	Let $\varGamma$ be a finite Abelian group, and for each $n\geq 1$ denote by $\wvG{n}$ 
	the quotient of $\vG \times \cdots \times \vG$ ($n$ copies) by its diagonal
	subgroup $\Delta \vG^n := \{ (g,\dots, g):\ g\in \vG \}$.
	Consider $\wvG{n} \rtimes S_n$, the semidirect product of $\wvG{n}$ with the 
	symmetric group $S_n$.

	Our main results are the following: 
	We observe that the pair $(\wvG{n} \rtimes S_n, S_n)$ is a Gelfand pair
	(Lemma \ref{L:mainlemma}). 
	Focusing on the finite cyclic groups, $\vG = \Z_r$, $r\in \N$
	in Theorem \ref{T:main_theorem}, we find explicit descriptions of the irreducible 
	constituents of the multiplicity free representation
	$
	\Ind{ \wvG{n} \rtimes S_n}{S_n}{1}.
	$ 
	Note that, when $r=n+1$, the group $\wvG{n}$ is 
	the ``group of parking functions,'' $\Z^n_{n+1}/\Diag\Z_{n+1}^n$.	
	Let $\bigoplus_\alpha V_\alpha$ be the decomposition of 
	$\Ind{\wvG{n} \rtimes S_n}{S_n}{1}$ into irreducibles, and let  
	$C_n(q) = \sum_\alpha  q^{\dim V_\alpha}$. Let $D(n)$, $n\in \N$ denote the set 
		$$
		\{ (k_0,\dots, k_n) \in \N^{n+1} :\  \sum k_i = n,\ \text{and}\ n+1\ \text{divides}\ 
		\sum_{i=1}^{n} i k_i \}.
		$$	
	It follows from the corollaries of our main result that 
	\begin{align*}
	C_n(q) = \sum_{ \k \in D(n)} q^{n \choose k_0,k_1,\ldots, k_n},
	\end{align*}
	where ${n \choose k_0,k_1,\ldots, k_n}$ is a multinomial coefficient. 
	Hence, 
	\begin{align*}
	C_n (1) = \frac{1}{n+1} {2n \choose n},\ \text{and}\ \frac{d C_n(q)}{dq} \Bigg|_{q=1} = (n+1)^{n-1}.
	\end{align*}
	It is interesting that 
	$D(n)$ can be identified with the set of polynomials $p(x) \in \Z[x]$ with non-negative integer 
	coefficients such that $p(1) =n$, and $p'(1)$ is divisible by $n+1$.

	Let $S(n)$ denote the set of all integer sequences $\mf{b}=(b_1,\dots,b_n) \in \N^n$ 
	satisfying $\sum_{i=1}^j b_i \geq j$ for all $j=1,\dots,n$, and 
	$\sum_{i=1}^n b_i =n$.
	It is well known that $\# S(n) = C_n(1)$. See $(o^5)$ in \cite{StanleyAddendum}.
	Define, 
	\begin{align*}
	S_n(q) = \sum_{ \mf{b} \in S(n)} q^{n \choose b_1,\ldots, b_n}.
	\end{align*}
	\begin{Conjecture}
	The polynomials $C_n(q)$ and $S_n (q)$ are identical for all $n\geq 1$.
	\end{Conjecture}
	The conjecture has been verified for $n\leq 10$.

	We organized our paper as follows. In Section \ref{sec-1}
	we introduce our notation and background. 
	In Section \ref{sec-2} we prove our main theorem. 
	In Section \ref{sec-3} we introduce and study a $q$-analogue of the Catalan numbers.
	We end our article with final remarks and open questions 
	in Section \ref{sec-final}.

	\section{\textbf{Preliminaries}}
	\label{sec-1}

	It is clear from the definition that the 
	symmetric group on $n$ letters $S_n$ acts on the set of parking functions of length $n$.
	For $n=3$, under the $S_3$ action, the set of parking functions divided into $5$ orbits:

{\bf
111

112, 121, 211

113, 131, 311

122, 212, 221
 
123, 132, 213, 231, 312, 321}.\\
	In each line, the first entry is chosen to be non-decreasing.
	We observe that orbits can be parameterized in terms of 
	non-decreasing parking functions. More generally, 
	non-decreasing integer sequences $(b_1,\dots,b_n)$ such that 
	$1\leq b_i \leq i$ for all $i=1,\dots,n$ is counted by the $n$-th Catalan number, 
	$$
	C_n := \frac{1}{n+1}{ 2n \choose n }.
	$$

	\subsection{Permutation Representation}
	Given a group $G$ and an action of $G$ on some set $X$, 
	the space of complex valued functions $\L(X)$ on $X$ becomes 
	a representation of $G$:
	Given $g\in G$ and $\psi: X \to \C$, define the action of $G$ on $\L(X)$
	as follows:
	$$ 
	(g\cdot \psi)(x) := \psi(g^{-1}\cdot x).
	$$
	The representation $\L(X)$ induced from the action of $G$ on the set $X$ 
	is called \emph{a permutation representation}.
	When $X$ itself is a vector space $V$ over complex numbers, $\L(V)$ is  
	the dual space, the action is the \emph{contragradient action}, and will be also denoted
	by $V^\vee$.


	From now on, assume that $G$ and $X$ are finite. 
	Then, one can consider the vector space $\C\cdot X$
	generated on the elements of $X$, and $G$ acts on $\C\cdot X$
	by  
	$$ 
	g \cdot (\sum_{x\in X} a_x x ) := \sum_{x \in X} a_x (g\cdot x),\ g\in G.
	$$
        Any $G$-representation $V$
	attains a $G$-invariant inner product, by averaging a given inner product
	over the group $G$. Consequently, any such 
	$V$ is isomorphic to its dual $V^\vee$ 
	as a $G$-representation.
	For this reason, the representations $\C\cdot X$ and $\L(X)$ are isomorphic
	when $G$ and $X$ are finite. In fact, the isomorphism is given by: 
	$$\begin{array}{ccc}
	\C\cdot X & \to & \L(X) \\
	x & \mapsto & \delta_x,
	\end{array}
	$$
	where $\delta_x$ is the Dirac delta function, $\delta_x(y)=1$ only if $x=y$
	for $x,y\in X$.

	When $G$ acts on a set $X$, then $X$ is decomposed into $G$-orbits.
	Similarly, the representation $\L(X)$ is decomposed into
	irreducible representations. However, by virtue of linearity,
	the decomposition of $\L(X)$ is
	much finer than that of the set $X$.

	Denote the set of orbits of $X$ by $\Orb(X)$. Each $G$-orbit is stable
	under the $G$-action, that is if $x\in O$ for some orbit $O$, 
	then for any $g\in G$,
	$g\cdot x\in O$. Therefore to each orbit we can attach a representation
	$\L(O)$. Since 
	$$
	X = \bigsqcup_{O \in \Orb(X)} O,
	$$
	the representation $\L(X)$ decomposes as
	$
	\L(X) = \oplus_{O \in \Orb(X)} \L(O).
	$
	However, typically one can decompose $\L(X)$ further by decomposing the
	representations $\L(O)$ corresponding to the orbits $O \in \Orb(X)$.
	By $\#Y$ denote the number of elements of a finite set $Y$, then
	$$ 
	\dim \L(O) = \, \text{Number of elements of }\, O = \#O.
	$$

	\subsection{Parking Function Module}
	When $X$ is the set of parkings function of length $n$, 
	we call the representation $\L(X)$ the \emph{ parking function module}
	and denote it by $\pf(n)$.

	Let's go back to the set of parking function of length $3$. 
	On this set, the symmetric group with $S_3$ acts and divides into $5$ orbits.
	We may parameterize the orbits of $S_n$ on $\pf(n)$ by the non-decreasing 
	elements they contain. Then,
	$$
	\pf(3) = \L({\bf 111}) 
	\oplus \L({\bf 112})
	\oplus \L({\bf 113})
	\oplus \L({\bf 122})
	\oplus \L({\bf 123}).
	$$

	It is well known that $S_3$ has $3$ irreducible representations,
	the {\em trivial representation} (of dimension $1$) which acts trivially on the
	one dimensional complex space, the {\em sign representation} 
	(of dimension $1$) which acts by the sign of the permutation on the one 
	dimensional complex vector space and a $2$-dimensional representation. 

	The symmetric group $S_3$ acts on $3$ letters ($=X$), 
	hence on the corresponding
	function space $\L(X)$, called the \emph{standard representation}. 
	The function $\delta_1 + \delta_2 + \delta_3$ is 
	invariant under $S_3$-action and form a one dimensional representation 
	(isomorphic to the trivial representation).  Let 
	$V$ be the orthogonal subspace to $\delta_1 + \delta_2 + \delta_3$. Then,
	$V$ is a two dimensional irreducible representation. In the basis
	$\delta_1, \delta_2, \delta_3$, the subspace $V$ is
	$$ V = \{ \sum a_i \delta_i : \sum a_i = 0 \}.$$

	Of the list above, $\L({\bf 111})$ is isomorphic to the trivial representation,  
	the three dimensional representations 
	$\L({\bf 112}), \L({\bf 113}), \L({\bf 122})$ are isomorphic to the standard
	representation (hence decompose as the sum of trivial representation and 
	the two dimensional representation), 
	$\L({\bf 123})$ is the regular representation of the group $S_3$.

	Given a representation $V$ and an irreducible representation of $V_i$ of a group
	$G$, the number of times the irreducible representation $V_i$ appears 
	in the decomposition of the representation $V$ 
	is called the {\em multiplicity} of $V_i$ in $V$.

	We summarize what we have seen in a table of multiplicities:
	$$
	\begin{array}{lccc}
	 & triv & sign & 2-dim \\
	\L({\bf 111}) & 1 & & \\
	\L({\bf 112}) & 1 & & 1 \\
	\L({\bf 113})  & 1 & & 1 \\
	\L({\bf 122}) & 1 & & 1 \\
	\L({\bf 123})  & 1 & 1 & 2.
	\end{array}
	$$

	Observe that the trivial representation appears in each representation
	$\L(O)$ for \emph{once}. 
	We state without proof that it is true in general.
	
	\begin{Lemma} \label{L:Trivial_in_PF}
	For all $n\geq 1$, the trivial representation of $S_n$ appears 
	in each orbit representation $\L(O)$ in $\pf(n)$ exactly once.
	\end{Lemma}

\subsection{Intertwining maps and Multiplicity Free Representations}

	A representation $V$ of $G$ is called \emph{multiplicity free} if 
	no irreducible representation of $G$ appears more than once in $V$.

	A $\C$-linear map $T: V \to W$ between two representations $V$ and $W$ of
	$G$ is called \emph{intertwining}, if 
	$$ 
	T(g\cdot v) = g \cdot T(v)
	$$
	for any $g\in G$ and $v\in V$. 
	If $V$ and $W$ are irreducible, then there are two cases:
	\begin{enumerate}
	\item $V$ and $W$ are non-isomorphic, and $T=0$,
	\item $V$ and $W$ are isomorphic, and $T=\lambda \id_V$ for some 
	$\lambda \in \C$.
	\end{enumerate}

	This is called the \emph{Schur's Lemma}.
	It follows that if $V=\bigoplus V_i^{m_i},\ m_i \in \Z_{\geq 0}$ is the decomposition of $V$ 
	into irreducibles, then the set of intertwining operators on $V$, 
	denoted by $\homself_G(V)$ is
	$$ 
	\homself_G (V) = \homself_G ( \bigoplus V_i^{m_i} ) = \bigoplus \homself( \C^{m_i} ).
	$$
	Therefore, $V$ is multiplicity free if and only if $m_i \leq 1$ if and only
	if the algebra $\homself_G(V)$ is commutative.

	\subsection{Gelfand Pairs}
	Given a finite group $G$ and a subgroup $K$, the group $G$ acts on the left
	coset space $G/K$. The pair $(G, K)$ is called a \emph{Gelfand pair} if
	the $G$-representation $\L(G/K)$ is multiplicity free.

	Given two subgroups $K_1$ and $K_2$ of $G$, the space of functions 
	$\L(K_1\backslash G/K_2)$
	on the double coset space $K_1\backslash G / K_2$ can be identified with the space of $K_1 	\times K_2$-invariant
	functions on $G$, consisting of those function $\psi: G \to \C$ so that 
	$\psi(k_1 g k_2)=	\psi(g)$ for all $(k_1, k_2) \in K_1 \times K_2$ and $g \in G$.

	The space of intertwining operators $\homself_G \L(G/K)$ on $\L(G/K)$ coincides with  
	$\L(K\backslash G/K)$.
	The algebra $\homself_G \L(G/K)$ is called the \emph{Hecke algebra} of the pair $(G,K)$.
	It follows from Schur's Lemma that the pair $(G,K)$ is a Gelfand pair if and only if 
	its Hecke algebra is commutative.

	The Lemma \ref{L:Trivial_in_PF} can be paraphrased as follows:  
	\begin{Lemma}
	The pair $(G,K)$ is a Gelfand pair for $G=\pf(n) \ltimes S_n$ and $K=S_n$. 
	Here, we use the identification of $\pf(n)$ with the Abelian group 
	$\Z^n_{n+1}/\Diag \Z_{n+1}^n$. The symmetric group $S_n$ acts on the coordinates 
	of $\pf(n)$.
	\end{Lemma}

	Equivalently, the pair $(G', K')$ is 
	a Gelfand pair, where $G' = \Z^n_{n+1} \ltimes S_n$ 
	and $K'=\Diag \Z_{n+1}^n \ltimes S_n$. Indeed, the symmetric group
	$S_n$ fixes the elements on the diagonal $\Diag  \Z_{n+1}$ point-wise,
	hence this action is trivial, and $K' \cong \Diag  \Z_{n+1}^n \times S_n$.
	The group $G'$ is an example of wreath product of finite groups:
	$G'= \Z^n_{n+1} \ltimes S_n$ is the wreath product of 
	$ \Z^n_{n+1}$ by $S_n$.

	\subsection{Wreath Products}
	Let $\vG$ be a finite group. The symmetric group $S_n$ acts on 
	$\vG^n$ by permuting the coordinates. 
	Form the semi-direct product $\vG^n \ltimes S_n$. The
	resulting finite group is called the \emph{wreath product} of $\vG$
	by $S_n$ and denoted by $\vG \wr S_n$. 
	An element of $\vG\wr S_n$ is denoted by the symbol 
	$(g_1, \ldots, g_n; \sigma)$, where $g_i \in \vG$ and $\sigma \in S_n$.
	The product of two elements $(g_1, \ldots, g_n; \sigma)$ and 
	$(h_1, \ldots, h_n; \tau)$ is
	$$(g_1, \ldots, g_n; \sigma)(h_1, \ldots, h_n; \tau) = 
	(g_1 h_{\sigma\inv(1)}, \ldots, g_n h_{\sigma\inv(n)}; \sigma\tau).$$

	We have previously stated that the pair 
	$(\Z_{n+1} \wr S_n, \Diag \Z_{n+1}^n \ltimes S_n)$ is Gelfand. 
	Inspired by this observation, we can ask for any group $\vG$:

	\vspace{.3cm}
	\noindent{\sc Question.}
	Is $(\vG \wr S_n, \Diag \vG^n \times S_n)$ a Gelfand pair?
	\vspace{.3cm}

	When $n=2$, the answer is ``yes'' for any group $\vG$. For $n\geq 3$, the
	results are mixed. Generally, the answer is no.
	We have tested this claim using \GAP. Some of the results are:
	$$
	\begin{array}{ccc}
	\text{Group} \vG & \text{true} & \text{false} \\
	\hline 
	S_3 & n\leq 5 & n=6 \\
	A_4 & n\leq 3 & n=4 \\
	GL(2, {\F}_3) & n=2 & n=3 \\
	SL(3, {\F}_2) & n=2 & n=3
	\end{array}
	$$

	In general, one can view being a Gelfand pair for $(G,K)$ as
	a measure of distance between the group $G$ and 
	its subgroup $K$. In some sense, for a Gelfand 
	pair $(G,K)$, the distance between the subgroup $K$ and the whole
	group $G$ is not that big. In the examples above, as $n$ gets larger, 
	the group $G=\vG \wr S_n$ gets much larger than the subgroup 
	$K=\vG \times S_n$. We believe that this is why the 
	answer is typically ``no.''

	\section{\textbf{Generalized Parking Functions}}
	\label{sec-2}

	Let $\vG$ be a finite group. It is known that 
	the pair $(\vG \wr S_n, S_n)$ is Gelfand, in the case of finite cyclic group
	$\Z_r$ by \cite{Mizukawa}, and in the case of arbitrary Abelian group
	$\vG$ by  \cite{Ceccherini}.

	For the rest of article we denote  by $\vG$  a finite Abelian group and  
	by $\Delta \vG^n$ the copy of $\vG$ sitting inside $\vG^n$ diagonally,
	and by $\wvG{n}$ the quotient group $\vG^n / \Delta \vG^n$.
	Finally, the image of an element $g\in \vG^n$ in $\wvG{n}$ is denoted by 
	$\widetilde{g}$

	The symmetric group $S_n$ acts naturally on the product $\vG^n$ and 
	hence on the  leaving the quotient group $\vG^n / \Delta \vG^n$.  
	Denote this $S_n$-representation by $V_n$. 
	There is a natural projection 
	\begin{align}\label{A:isom}
	\pi: \vG^n \rtimes S_n &\rightarrow \wvG{n} \rtimes S_n\\   
	(g,\sigma) &\mapsto (\widetilde{g}, \sigma). \notag
	\end{align}

	Denote the identity element of the symmetric group $S_n$ by $e_0$.
	\begin{Lemma}
	 The group
	 $\wvG{n} \rtimes S_n$ is isomorphic to 
	 $(\vG^n \rtimes S_n)/ (\Delta \vG^n \rtimes \{ e_0\} )$.
	\end{Lemma}

	\begin{proof}
	Let $\pi :  \vG^n \rtimes S_n \rightarrow \wvG{n} \rtimes S_n$ 
	be as in (\ref{A:isom}).
	It is clear that $\pi$ is surjective. 
	We check that $\pi$ is a homomorphism, indeed:
	\begin{align*}
	\pi (( g,\sigma) (h, \tau)) &= \pi ( (g \sigma (h), \sigma \tau ))\\
	&=  ( \widetilde{g \sigma (h)}, \sigma \tau ) \\
	&= ( \widetilde{g}, \sigma ) ( \widetilde{h}, \tau).
	\end{align*}

	Suppose now that $\pi (g,\sigma) = (\widetilde{0}, e_0)$. 
	Then, $\sigma = e_0$ and $g \in \Delta \vG^n$. 
	In other words, the kernel of $\pi$ is $\Delta \vG^n \rtimes \{ e_0 \}$. 
	In particular, $\Delta \vG^n \rtimes \{e_0\}$ is normal in $\vG^n \rtimes S_n$.
	Since $\vG^n$ is Abelian, there is a splitting  
	$\vG^n = H \cdot \Delta \vG^n$ where $H \cap \Delta \vG^n = \{ 0 \}$, 
	and $H$ is isomorphic to $\vG^n / \Delta \vG^n$. 
	Since 
	$$
	(H \rtimes S_n ) \cdot (\Delta \vG^n \rtimes {e_0})  = 
	(\Delta \vG^n \cdot H) \rtimes  S_n = \vG^n \rtimes S_n,$$
	it follows that $H \rtimes S_n \cong \vG^n / \Delta \vG^n \rtimes S_n.$ 
	In other words, $\vG^n \rtimes S_n$ is the semidirect 
	product of $\Delta \vG^n \rtimes S_n$ and 
	$\wvG{n} \rtimes S_n$.

	\end{proof}

	\begin{Lemma}\label{L:mainlemma}
	$(\wvG{n} \rtimes S_n, S_n)$ is a Gelfand pair. 
	\end{Lemma}

	This follows from \cite{Ceccherini}.

\subsection{$\vG$ is a finite cyclic group, $\Z_r$}

	The irreducible representations of the group $\vG \wr S_n = \Z_r \wr S_n$ 
	are described as follows (see \cite{ATY}):
	Let $\k=(k_0, \ldots, k_{r-1})$ be an $r$-tuple of nonnegative integers, 
	and let $\lambda = \lambda_{\k} = (\lambda_1,\ldots,\lambda_\ell)$ 
	be the partition $(0^{k_0}, 1^{k_1}, \ldots, (r-1)^{k_{r-1}} )$ \footnote{In this notation, 
	the number of parts of $\lambda$ that is equal to $i$ is $k_i$.}.
	Suppose
	$$
	 \sum_{i=0}^{r-1} k_i= n.
	$$ 
	Then, the length $\ell$ 
	of $\lambda$ equals $n$, i.e. $n=\ell(\lambda)=\ell$.

	Let $x_1,\dots, x_n$ be algebraically independent variables, and let  
	$$
	M(\k)= \{ x_{\sigma (1)}^{\lambda_1} x_{\sigma(2)}^{\lambda_2} \cdots 
	x_{\sigma (n)}^{\lambda_n}:\ \sigma \in S_n,\ \lambda = \lambda_\k \}.
	$$ 
	We denote by $\C M(\k)$ the vector space on $M(\k)$. 
	Then, $\vG \wr S_n$ acts on $\C M(\k)$ by 
	$$
	(g, \sigma ) \cdot f(x_1,\ldots , x_n) := 
	f ( g_1^{-1} x_{\sigma(1)} , \ldots , g_n^{-1} x_{\sigma (n) }),
	$$
	where $g= (g_1,\ldots , g_n) \in \vG^n$, $\sigma \in S_n$, and $f\in \C M(\k)$. 
	It is shown in \cite{ATY} that  $\C M(\k)$ is an irreducible $\vG \wr S_n-$module
	and all irreducible representations of $\vG \wr S_n$ can be obtained this way.

	In \cite{Mizukawa} it is shown that  
	$$ 
	\Ind{\vG \wr S_n}{S_n}{1}= \bigoplus_{\sum k_i = n } \C M (\k ).
	$$


	\begin{Remark}
	Let $(g,\sigma) \in \vG \wr S_n$  and  
	$x_{\tau(1)}^{\lambda_1} \cdots x_{\tau(n)}^{\lambda_n} \in M(\k)$.
	Then, 
	\begin{align*}
	(g,\sigma) \cdot (x_{\tau(1)}^{\lambda_1} \cdots x_{\tau(n)}^{\lambda_n}) &= 
	(g,\sigma) \cdot (x_1^{\lambda_{\tau^{-1}(1)}} \cdots x_{n}^{\lambda_{\tau^{-1}(n)}}) \\
	&=
	\frac{1}{ \prod_i g_{\sigma(i)}^{\lambda_{\tau^{-1}(i)} } }
	x_{\sigma(1)}^{\lambda_{\tau^{-1}(1)}} \cdots x_{\sigma(n)}^{\lambda_{\tau^{-1}(n)} } \\
	&= \frac{1}{ \prod_i g_{\sigma(\tau(i))}^{\lambda_i } }
	x_{\sigma(\tau(1))}^{\lambda_{1}} \cdots x_{\sigma(\tau(n))}^{\lambda_{n} }.
	\end{align*}
	Therefore, the matrix of the representation $\Mk$ at $(g,\sigma) \in \vG \wr S_n$ is 
	a monomial matrix and its unique non-zero diagonal entry is 
	$ 1/ \prod_i g_{\sigma(i)}^{\lambda_i } $.
	It follows that the character  $\chi_{\Mk}$ of $\Mk$ evaluated at $(g, \sigma)$ equals 
	\begin{equation}\label{E:characterofMk}
	 \chi_{\Mk} (g,\sigma) = \frac{1}{  g_{\sigma(1)}^{\lambda_1 } 
	\cdots g_{\sigma(n)}^{\lambda_n}}. 
	\end{equation}
	\end{Remark}

\subsection{Zonal Spherical Functions}

	Let $(G,K)$ be a Gelfand pair. Suppose that 
	$\bigoplus^s_{i=1} V_i$ is the decomposition of $\Ind{G}{K}{1} $ 
	into irreducible subrepresentations. 
	For $1 \leq i \leq s$, let $\chi_i$ denote the character of $V_i$ and define 
	the $\C-$valued function 
	\begin{equation}
	\omega_i (x) =   \frac{1}{| K |} \sum_{h\in K} \chi_i(x^{-1}h),\ \text{for}\, x\in G.
	 \end{equation}
	The functions $\omega_i$ for $i = 1\leq i \leq s$ are called the 
	{\em zonal spherical functions} of the pair $(G,K)$.
	They form an orthogonal basis of $\homself_G \L(G/K)$.

	Let $m_\lambda (x_1,\dots, x_n )$ be the {\em monomial symmetric function}
	defined as
	$$
	m_{\lambda}(x_1,\ldots, x_n) = \frac{1}{k_0 ! k_1 ! \cdots k_n! } 
	\sum_{\sigma \in S_n} x_{\sigma(1)}^{\lambda_1} \cdots x_{\sigma(n)}^{\lambda_n}.
	$$
	In \cite{Mizukawa}, it shown that the zonal spherical function $\omega^\k$ 
	corresponding to an irreducible constituent $\C M(\k)$ of $\Ind{\vG \wr S_n}{S_n}{1}$
	can be expressed  in terms of the monomial symmetric function $m_\lambda$:
	$$
	\omega^{\k} (( h ,\sigma )) = 
	\frac{ m_{\lambda} (h_1,
	\ldots, h_n)}{ m_{\lambda} (1,\ldots, 1)} ,\	 \text{for}\,(h,\sigma) \in \vG \wr S_n,
	$$
	where $\lambda$ is the partition determined by $\k$.

	\begin{Remark}
	It is easy to see that 
	$m_{\lambda}(1,\ldots, 1) = {n \choose k_1 ,\ldots, k_n}$.
	\end{Remark}

	\subsection{Zonal spherical functions of $(\wvG{n} \rtimes S_n, S_n)$.}
	
	In the rest of the article, we will concentrate on 
	the {\em generalized parking function module}
	$$
	\Ind{ \wvG{n} \rtimes S_n}{S_n}{1}.
	$$ 
	We start by a general observation.

	\begin{Fact}\label{F:exercise} (See Section 1.13, Exercise 6.b in \cite{Sagan})
	Let $N$ be a normal subgroup of a finite group $G$, and let $Y$ is a representation of 
	$G/N$. Define a function $X$ on $G$ by $X(g) = Y( g  N)$. Then, 
	$X$ is an irreducible representation of $G$ if and only if $Y$ is an irreducible
	representation of $G/N$. In this case $X$ is said to be ``lifted from $Y$.''
	\end{Fact}

	Since $\wvG{n} \rtimes S_n$ is isomorphic to 
	$(\vG \wr S_n) / (\Delta \vG^n \rtimes \{ e_0\} )$, we can use 
	Fact \ref{F:exercise}: Let 
	$\Mk$ be an irreducible representation of $\vG \wr S_n$ for some 
	$\k= (k_0, \ldots, k_{r-1}) \in \Z^{r}_{\geq 0}$ where $\sum k_i =n$.
	The representation $\Mk$ descends to a representation $Y$ of $\wvG{n}$
	if and only if the normal subgroup $\Delta \vG^n \rtimes \{ e_0\} $
	acts trivially on the representation $V$. In which case, the representation $Y$
	is irreducible by the above fact.
	
	
	Let $(c,e_0) \in \Delta \vG^n \rtimes \{ e_0\} $  and  
	$x_{\tau(1)}^{\lambda_1} \cdots x_{\tau(n)}^{\lambda_n} \in M(\k)$, 
	where $\xi=e^{\frac{2\pi i}{r}}$ and
	$c=(\xi^j, \ldots, \xi^j)$ for some $j=0, \ldots, r-1$.
	
	We calculate the effect of $(c,e_0) \in \Delta \vG^n \rtimes \{ e_0\} $  on
	$x_{\tau(1)}^{\lambda_1} \cdots x_{\tau(n)}^{\lambda_n}$:
	
	$$(c,e_0) \cdot 
	x_{\tau(1)}^{\lambda_1} \cdots x_{\tau(n)}^{\lambda_n} = (\xi^{-j\sum \lambda_i})
	x_{\tau(1)}^{\lambda_1} \cdots x_{\tau(n)}^{\lambda_n}
	.$$

	The action of $\Delta \vG^n \rtimes \{ e_0\} $ on $\Mk$ is trivial if and only if
	the right hand side equals $x_{\tau(1)}^{\lambda_1} \cdots x_{\tau(n)}^{\lambda_n}$.
	Given that $\xi$ is a primitive $r^{th}$ root of unity, this is only possible if and only if
	$r$ divides the sum $\sum \lambda_i = \sum ik_i$.
	
	Thus, $\Mk$ is an irreducible representation of $\wvG{n} \rtimes S_n$ if and only if 
	$\k = (k_0,\ldots, k_{r-1})$ satisfies the following two conditions:
	\begin{itemize}
	\item $\sum_i k_i = n$,
	\item $r$ divides $ \sum i k_i$.
	\end{itemize}
	Our main result follows from these observations:

	\begin{Theorem} \label{T:main_theorem}
	Let $\vG= \Z_r$ and 
	$\k = (k_0,\ldots, k_{r-1})$ be a sequence of non-negative integers
	of length $r$. 
	Then, 
	\begin{enumerate}
	\item 
	The $\wvG{n} \rtimes S_n$-representation  $\C M(\k)$ is  irreducible  of 
	of degree $n \choose {k_0, \ldots, k_{r-1}}$ if and only if
	$\sum_i k_i = n$ and $r$ divides $ \sum i k_i$.

	\item The pair $(\wvG{n} \rtimes S_n, S_n)$ is Gelfand, and furthermore, 
	$\Ind{\wvG{n} \rtimes S_n}{S_n}{1}$ decomposes into irreducibles as  follows:
	\begin{align}\label{A:DecompositionCatalan}
	\Ind{\wvG{n} \rtimes S_n}{S_n}{1} = 
	\bigoplus_{\stackrel{\sum k_i = n}{ r \big{|} \sum i k_i } } \C M (\k ).
	\end{align}
	
	\end{enumerate}
	\end{Theorem}

	\begin{Remark}
	Let $g\in \vG^n$ and $\widetilde{g}$ its image in $\wvG{n}$.
	A simple calculation shows that the double coset $S_n ( \widetilde{g}, e_0 ) S_n $
	in $\wvG{n} \rtimes S_n$ 
	is  $\{ (\sigma (\widetilde{g}), \sigma \tau):\ \sigma, \tau \in S_n \}$.
	Therefore $ \# S_n ( \widetilde{g}, e_0 ) S_n $
	is the cardinality of the $S_n$-orbit $S_n\cdot \widetilde{g}$ in $\wvG{n}$  times
	$n!$.
	Furthermore, the number of double cosets of $S_n$ in $\wvG{n} \rtimes S_n$ is the number 
	of $S_n-$orbits in $\wvG{n}$.
	\end{Remark}

	\begin{Corollary}\label{C:r_to_the_n}
	For any positive integer $n$ and $r$ the following identity holds
	$$ r^{n-1} = 
	\sum {n \choose {k_0, \ldots, k_{r-1}}},
	$$
	where the summation is over all nonnegative integer sequences
	of length $r$ such that 
	${\sum_i k_i = n}$ and $r$ divides $\sum_i i k_i$. 
	\end{Corollary}

	\section{\textbf{A $q-$analogue of the Catalan Numbers.}}
	\label{sec-3}

	Set $r=n+1$ and  $\vG= \Z_{n+1}$. Recall that 
	as an $S_n$-set $\wvG{n}$  is isomorphic to the set of parking functions. Recall
	also that 
	\begin{Fact}\label{F:numberoforbits}
	The number of 
	$S_n$-orbits in an $S_n$-set $X$ is the multiplicity of the trivial representation
	in the associated representation $\L(X)$.
	\end{Fact}

	\begin{Proposition}\label{P:Cats}
	Let $\vG = \Z_{n+1}$. Then,
	the number of irreducible representations 
	in $$\Ind{\wvG{n} \rtimes S_n} {S_n}{1}$$ 
	is the n-th Catalan number, $C_n$.
	\end{Proposition}
	\begin{proof}
	Immediate from the Fact \ref{F:numberoforbits}.
	\end{proof}

	Let $\bigoplus_\alpha V_\alpha$ be the decomposition of 
	$\Ind{\wvG{n} \rtimes S_n}{S_n}{1}$ as in (\ref{A:DecompositionCatalan}) where
	$r= n+1$. 
	Define 
	\begin{equation}\label{E:qcat}
	C_n(q) = \sum_\alpha  q^{\dim V_\alpha}.
	\end{equation}
	Since the multiplicity of each irreducible representation 
	in $\Ind{\wvG{n} \rtimes S_n}{S_n}{1}$ is 1, 
	the coefficients of the polynomial $C_n(q)$
	count the number of occurrences of the irreducible representations of a given dimension
	in $\Ind{\wvG{n} \rtimes S_n}{S_n}{1}$.

	Let $D(n)$ denote the set of all $n+1$ tuples $\k = (k_0,\dots, k_n) \in \N^{n+1}$ 
	such that $\sum_{i=0}^n k_i = n$ and $n+1$ divides $\sum_{i=1}^{n} i k_i $. 
	It follows from Theorem \ref{T:main_theorem} that 
	\begin{align}\label{A:cats_multinomials}
	C_n(q) = \sum_{ \k \in D(n)} q^{n \choose k_0,k_1,\ldots, k_n}.
	\end{align}
	It follows from Proposition \ref{P:Cats} that 
	\begin{align}\label{A:catat1}
	C_n (1) = C_n= \frac{1}{n+1} {2n \choose n},
	\end{align}
	and that
	\begin{align}\label{A:parkat1}
	\frac{d C_n(q)}{dq} \Bigg|_{q=1} = (n+1)^{n-1}.
	\end{align}

		\begin{Example}
		 There are 5 sequences of the form $(k_0,\dots, k_3) \in \N^4$
		 such that $k_0 + \ldots + k_3 = 3$ and $\sum_{i=1}^3 i k_i $ is divisible by 4. Namely, 
		 $D(3) = \{ \bf{ 3000, 1101, 0210,1020, 0012} \}.$	 
		 Then 
		 $
		 C_3(q) = q + 3 q ^3 + q^6.
		 $ 
		 Similarly, one can compute 
		 $
		 C_4 (q) = q + 4 q^4 + 2 q^6 + 6 q^{12} + q^{24}.
		 $
		\end{Example}

	Let $E(n)$ be the set of $n$ element multisets on $\Z_{n+1}$ 
	whose elements sum to 0. 
	It is known that the cardinality $E(n)$ 
	is equal to the $n$-th Catalan number. See \textbf{jjj} in 
	page 264 of \cite{Stanley2}.   
	Recently Tewodros Amdeberhan constructed an explicit 
	bijection between the sets $D(n)$ and $E(n)$.

	Set $C_n(q) = \alpha_1(n) q^{a_1}+  \alpha_2(n) q^{a_2} + \cdots + \alpha_m(n) q^{a_m}$,  
	where $\alpha_i(n) \neq 0$ for $i=1,\ldots , m$ and $1\leq a_1 \leq a_2 \leq \cdots \leq a_m$. 
	It is desirable to find a combinatorial interpretation of the coefficients $\alpha_i (n)$, 
	$1\leq i \leq m$.

	Let $t_0 < t_1 < t_2 < \cdots $ be a collection of ordered, commuting variables.
	Let $\sigma$ be a plane tree (connected graph with no cycles) on $n$ vertices. 
	Set
	$$
	t^\sigma = \prod_{i \geq 0} t_i^{d_i(\sigma)},
	$$
	where $d_i(\sigma)$ is the number of vertices of $\sigma$ of degree $i$. 
	Define 
	$$
	s_n = \sum_{\tau} t^\tau,
	$$ 
	summed over all plane trees with $n$ vertices. For example, 
	$s_4 = t_0^3 t_3 + 3 t_0^2 t_1 t_2 + t_0 t_1^3$. 
	
	Using lexicographic ordering on the monomials of $s_n$, 
	we denote by $v_n$ the sequence of coefficients of the polynomial $s_n$.
	For example, $v_4 = (1,3,1)$.
	It turns out, for $n\leq 6$ the sequence $v_{n+1}$ agrees with 
	the sequence $( \alpha_1(n), \cdots, \alpha_m(n))$ of coefficients of $C_n(q)$.
	However, 
	$$
	(\alpha_i (7) )_{i=1}^{14} = (1, 7, 7, 7, 21, 42, 21, 56, 105, 35, 35, 70, 21, 1),
	$$
	and $v_8 = (1, 7, 7, 7, 21, 42, 21, 21,35,105, 35, 35, 70, 21, 1).$
	\begin{Remark}
	Coefficient vectors $v_n,\ n\geq 1$ can be computed by the Lagrange inversion
	formula (see page 40 of \cite{Stanley2}). 
	\end{Remark}
	


	\section{ \textbf{Final Remarks and Questions}}
	\label{sec-final}


	Let $\k = (k_0,\dots, k_n)$ be an element of $D(n)$ 
	and let $\C M(\k)$ be the associated irreducible constituent
	of $\Ind{\wvG{n} \rtimes S_n}{S_n}{1}$.
	Let $\omega_{PF}^{\k}$ be  its the zonal spherical function.

	Given $(h,\sigma) \in (\vG^{n} \wr S_n, S_n)$, by $\widetilde{h}$ denote
	 the image of $h$ in $\wvG{n}$. 
	Then,
	$$
	\omega_{PF}^{\k} ( \widetilde{ h} ,\sigma) = 
	\frac{ m_{\lambda} (h_1,
	\ldots, h_n)}{ m_{\lambda} (1,\ldots, 1)} ,
	$$
	where $\lambda$ is the partition determined by $\k \in D(n)$.

	The values of the zonal spherical functions for $\wvG{n} \rtimes S_n$ do not need 
	to be real. For example, here is the picture of the 
	image of the zonal spherical function 
	$\omega^{\k} :\wvG{5} \rtimes S_5 \rightarrow \C$
	for $\k = (0,0,0,2,3) $:

		\begin{figure}[ht]
		\begin{center}
		\includegraphics[scale = .3]{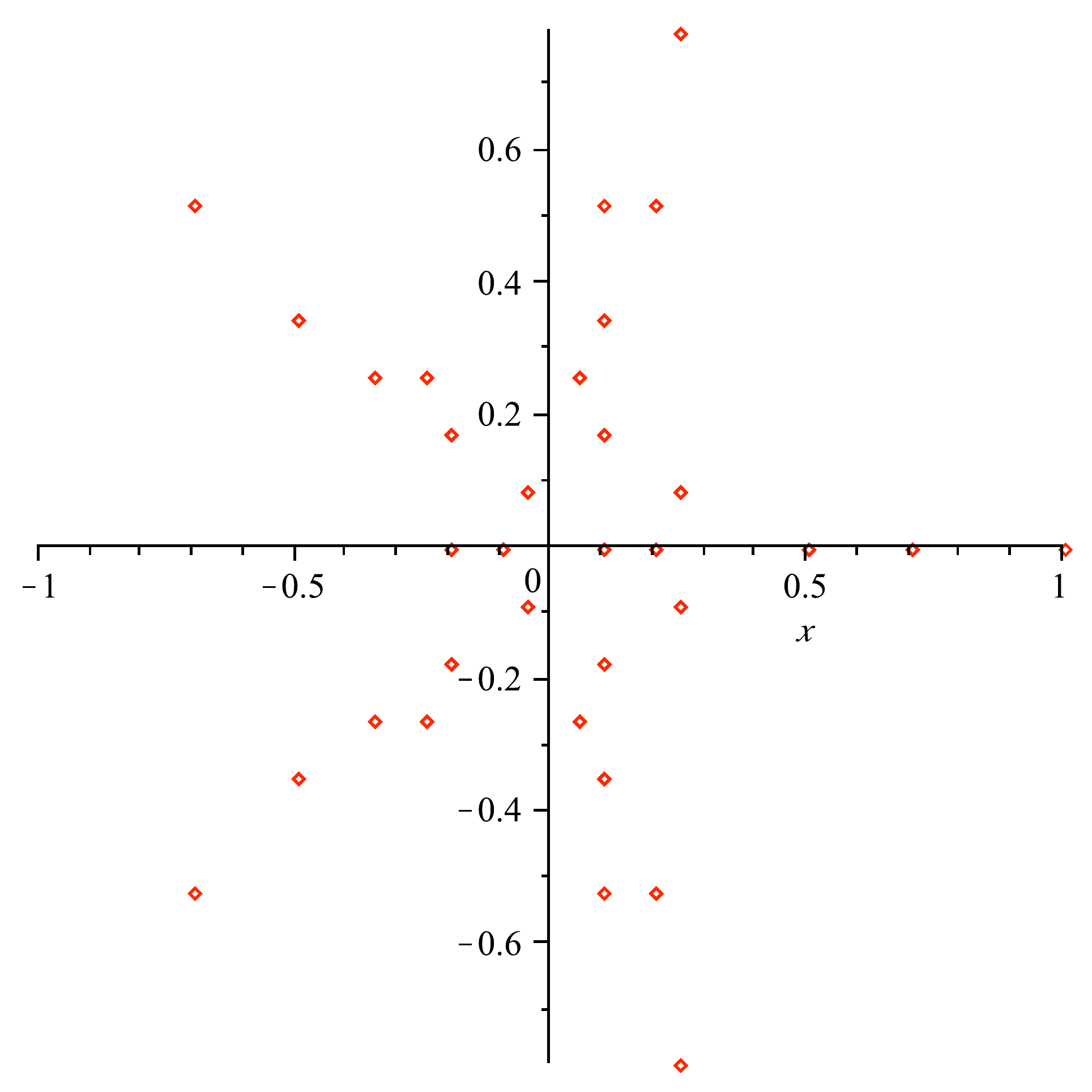}
		\caption{}
		\label{G:graphofsigma}
		\end{center}
		\end{figure}

	What about the number of real valued zonal spherical functions?
	A few terms of in this sequence are $2,3,6,10,\ldots$ (starting at $n \geq 2)$.
	Can we say anything about the injectivity of these functions? 
	How about their Fourier transforms?

	Let $\{ x_1,y_1,\dots, x_n,y_n \}$ be 
	a set of $2n$ algebraically independent variables on which $\sigma \in S_n$ acts by 
	$$
	\sigma \cdot x_i = x_{\sigma^{-1}(i)},\ 
	\sigma \cdot y_i = y_{\sigma^{-1}(i)}\ \text{for}\ i=1,\dots,n.
	$$
	For non-negative integers $r,s \in \N$, define 
	$
	p_{r,s} = \sum_{i=1}^n x_i ^r y_i ^s.
	$
	The ring of diagonal co-invariants (of $S_n$) is the 
	$S_n$-module 
	$$
	R_n = \C [ x_1,y_1,\ldots, x_n,y_n] / I_+,
	$$ 
	where $I_+$ is the 
	ideal generated by the power sums 
	$p_{r,s}$ with $r+s > 0$.

	Let $\vG = \Z_{n+1}$. 
	It is known that $R_n$ is isomorphic as an
	$S_n$-module to 
	$(\vG^n / \Delta \vG^n) \otimes sign$  (see \cite{Haiman}).
	Notice that $R_n$ is a bi-graded $S_n$-module.  
	Let $\mathcal{F}(q,t)$ be the corresponding bi-graded character.
	The multiplicity of the sign character in $\mathcal{F}(q,t)$ is 
	the so called {\em $q,t-$Catalan series}.
	What is the relationship
	between the $q,t-$Catalan series and our $q$ analogue? 
	It would be interesting to investigate similar questions for the associated 
	``twisted Gelfand pair,'' also.  (For definition, see \cite{Mac}, page 398).

	\subsection{Acknowledgements} We thank Tewodros Amdeberhan 
	for helpful discussions on the Catalan items.

\end{document}